\documentclass[12pt]{amsart}
\usepackage[margin=1in]{geometry}
\usepackage[english]{babel}
\usepackage[utf8]{inputenc}
\usepackage{subcaption}
\usepackage{amsmath}
\usepackage{amssymb}
\usepackage{amsfonts}
\usepackage{amsthm}
\usepackage{mathrsfs}
\usepackage[all]{xy}
\usepackage[pdftex]{graphicx}
\usepackage{color}
\usepackage{cite}
\usepackage{url}
\usepackage{indent first}
\usepackage[labelfont=bf,labelsep=period,justification=raggedright]{caption}
\usepackage[english]{babel}
\usepackage[utf8]{inputenc}
\usepackage{hyperref}
\usepackage[colorinlistoftodos]{todonotes}
\usepackage{tkz-fct}
\usepackage{tikz}
\usetikzlibrary{calc}

\newcommand{\NN}{\mathbb{N}}

\newcommand{\mcN}{\mathcal{N}}

\newcommand{\mcP}{\mathcal{P}}

\newcommand{\msG}{\mathscr{G}}
\newcommand{\msP}{\mathscr{P}}

\newcommand{\Mem}{\textsc{Mem}}
\newcommand{\Memplus}{\textsc{Mem}$^+$}
\newcommand{\Memo}{\textsc{Mem}$^0$}

\theoremstyle{plain}
\newtheorem{thm}{Theorem}
\newtheorem{lemma}[thm]{Lemma}

\newtheorem{prop}[thm]{Proposition}

\newtheorem*{claim}{Claim}

\theoremstyle{definition}
\newtheorem{defn}[thm]{Definition}

\theoremstyle{remark}

\newtheorem*{ex}{Example}

\numberwithin{equation}{section}
\numberwithin{thm}{section}

\newcommand{\PreserveBackslash}[1]{\let\temp=\\#1\let\\=\temp}
\newcolumntype{C}{>{\PreserveBackslash\centering}p{2ex}}

\begin{document}

\begin{abstract}
Memgames are heap games in which the play constraints on a given heap $H$ are determined by the immediately preceding move on~$H$. We analyze  three related memgames, which we call \Mem, \Memplus, and \Memo, that have simple, parameterless definitions but that nonetheless exhibit intricate and surprising nim value structures. The paper concludes with a long list of open questions and intriguing directions for further research.
\end{abstract}

\title{Memgames}
\author{Urban Larsson}
\address{Urban Larsson, National University of Singapore, Singapore}
\email{urban031@gmail.com}
\author{Simon Rubinstein-Salzedo}
\address{Simon Rubinstein-Salzedo, Euler Circle, Mountain View, CA 94040, USA}
\email{simon@eulercircle.com}
\author{Aaron N.\ Siegel}
\address{Aaron N.\ Siegel, San Francisco, CA, USA}
\email{aaron.n.siegel@gmail.com}
\date{\today}
\maketitle

\section{Introduction}

Consider the following impartial combinatorial game, played on a single heap of tokens. On her first turn, the first player may remove any positive number of tokens (but at most the full heap). On subsequent turns, if $k$ tokens were removed on the immediately preceding turn, then the next player must remove {\it at least} $k$ tokens. If fewer than $k$ tokens remain, then the game ends. The tokens removed on each succeeding turn are therefore constrained to form a monotonically nondecreasing sequence. The winner is determined according to the usual normal play convention: the first player unable to move loses. Played on a single (nonempty) heap, of course, the game is trivial; the first player can win simply by removing the entire heap. 

Now consider the same game played on multiple heaps, with each heap $H$ maintaining its own independent ``memory'' of the number of tokens removed on the immediately preceding play on~$H$. Denote by $n_k$ a heap of size $n$ with memory~$k$; then a legal move is to remove some number $j$ of tokens from $n_k$, with $k \leq j \leq n$, leaving the position $(n-j)_j$. We denote by $n_0$ a starting-position heap (from which no tokens have ever been removed). Then succinctly,
\[
n_k = \{(n-j)_j : k \leq j \leq n\}.
\]

\begin{figure}
\centering
\begin{tabular}{ccc}
    \includegraphics[width=0.3\textwidth]{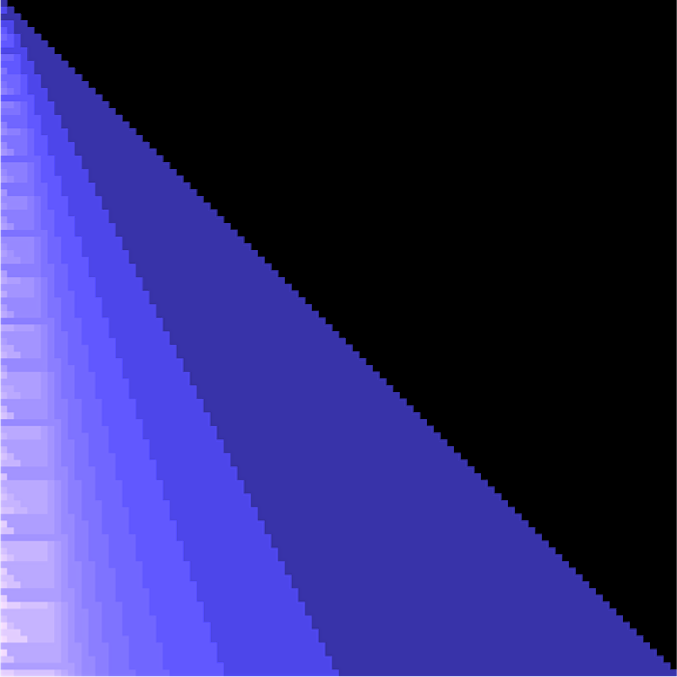} &
    \includegraphics[width=0.3\textwidth]{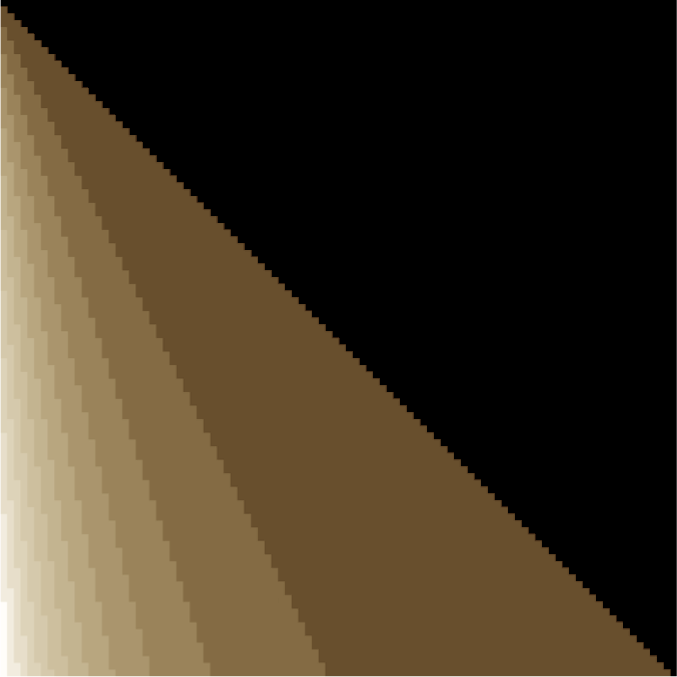} &
    \includegraphics[width=0.3\textwidth]{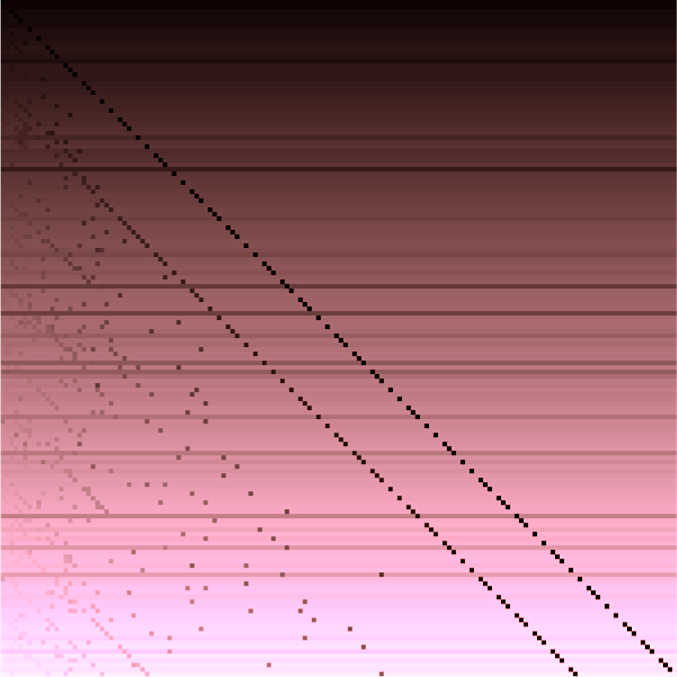} \\
    (a) \Mem & (b) \Memplus & (c) \Memo
\end{tabular}
\caption{``Heat maps'' of the nim values of \Mem, \Memplus, and \Memo. \label{fig:threeheatmaps}}
\end{figure}

We call this game \Mem, and despite its simple parameterless definition, it turns out to have an unexpectedly rich structure. See Figure~\ref{fig:threeheatmaps}(a) for a ``heat map'' of the nim values of \Mem. The nim value of $n_k$ appears at row~$n$, column~$k$ of the diagram, with lower nim values represented by darker shades, so that nim value $0$ is solid black. A striking quadratic structure is evident at first glance, and in fact we will prove shortly (Theorem~\ref{thm:memk2}) that the nim values satisfy
\[
\msG(n_k)=\left\lfloor\frac{n}{k}\right\rfloor \quad\text{whenever}\quad k^2 \ge n.
\]
When $k^2 < n$, however, the nim values of \Mem{} scatter into a mysterious fractal-like pattern, visible in a narrow vertical band along the left side of Figure~\ref{fig:threeheatmaps}(a). There is evidently an intriguing fine structure, which we will discuss briefly in Section \ref{section:openquestions}; but the region beneath the parabolic ``envelope'' remains poorly understood.

We will also consider two closely related games, whose heat maps are also pictured in Figure~\ref{fig:threeheatmaps}:
\begin{itemize}
    \item \Memplus: Remove $j$ tokens from a heap~$H$, where $j$ is {\it strictly greater} than the number removed on the immediately previous play on~$H$:
    \[
    n_k = \{(n-j)_j : k < j \leq n\}
    \]
    \item \Memo: Remove $j$ tokens from a heap~$H$, where $j$ is {\it not equal to} the number removed on the immediately previous play on~$H$:
    \[
    n_k = \{(n-j)_j : 1 \leq j \leq n,\ j \neq k\}
    \]
\end{itemize}

Intriguingly, the nim values of \Memplus{} are a sort of simplification of those of \Mem, with a similar quadratic structure, but without the added fractal-like complexity. In Section \ref{section:memplus} we will give a complete solution to \Memplus.

Perhaps most interesting of all is \Memo. It turns out that for any particular heap size~$n$, its nim values $\msG(n_k)$ are constant for all $k \geq n+1$. We denote this limiting value by $\msG(n_\infty)$, the \textbf{frontier value} at~$n$. The frontier values can be visualized as solid horizontal stripes above the lower triangular region in Figure~\ref{fig:threeheatmaps}(c). We will show that:
\begin{enumerate}
    \item[(i)] Every integer $m \geq 0$ is the frontier value of \emph{at least one} heap size~$n$ (that is, $m = \msG(n_\infty)$ for at least one~$n$).
    \item[(ii)] If $m$ is the frontier value of \emph{at least two} distinct heap sizes, then $m = \msG(n_k)$ for just finitely many values of~$n$ (the \textbf{Mortality Theorem}, Theorem~\ref{thm:mortality}).
    \item[(iii)] Conversely, suppose that $m$ is the frontier value of \emph{exactly one} heap size~$t$, so that $m = \msG(t_\infty)$. Then there are infinitely many $n$ for which some $k$ satisfies $m = \msG(n_k)$, and the positions of nim value $m$ concentrate on the diagonal $(t+k)_k$.
\end{enumerate}
The first two instances of condition (iii) occur at $m = 0$ and $m = 12$, giving rise to the two visible diagonals in Figure~\ref{fig:threeheatmaps}(c). The positions of nim value $0$ are exactly those of the form $k_k$ for which the dyadic valuation of $k$ is even. The positions of nim value $12$ concentrate on the diagonal $(22+k)_k$, and they too are closely related to the dyadic valuation of~$k$. Two further diagonals are known, at $m = 1270$ and $m = 105161$, and they are discussed in more detail in Section~\ref{section:memo_nim_values}.

The behavior of \Memo{} is deeply mysterious. It has a simple definition with no parameters: what, then, is special about $m = 12$, at which a second diagonal suddenly appears? Are there infinitely many such diagonals, and if so, can one characterize their nim values? These and other open questions are discussed in Section~\ref{section:openquestions}.

\subsection*{Generalizations and Related Games.}

\Mem, \Memplus, and \Memo{} are instances of a broader class of \textbf{memgames}. Given any function $F: \NN \to \mcP(\NN^+)$, where $\mcP(\NN^+)$ denotes the powerset of the (strictly) positive integers, we define a ruleset $\Gamma = \Mem(F)$ played with heaps of tokens, as follows. A single heap of $\Mem(F)$ has the form~$n_k$, with options given by 
\[
n_k = \{(n-j)_j : 1 \leq j \leq n,\ j \in F(k)\}.
\]
That is, if $k$ tokens were removed on the immediately preceding turn, then the next player may remove $j$ tokens if and only if $j \in F(k)$. An ``untouched'' heap of size~$n$, from which no tokens have yet been removed, is represented by~$n_0$, with permitted moves given by $F(0)$. When $\Mem(F)$ is played on multiple heaps, then each heap has its own independent ``memory''. We call $F$ the \textbf{memfunction} of~$\Gamma$. For example, the memfunction of \Memplus{} is given by $F(k) = \{j \in \NN^+: j > k\}$. 


Several classical rulesets can be characterized as memgames. For instance, if $F(k)=\{1,2,\ldots,2k\}$, then we recover the game of \textsc{Fibonacci Nim} \cite{Whinihan63}, with the exception that in \textsc{Fibonacci Nim}, the first player may not remove all the tokens. More generally, if, for some $\alpha\ge 1$, we have $F(k)=\{1,2,\ldots,\lfloor\alpha k\rfloor\}$, then we recover a class of take-away games studied, for instance, in~\cite{Schwenk70}, \cite{LRS15}, and~\cite{RS18}. A more general class of memgames was studied in~\cite{HHR03}. In all these papers, the memfunction $F$ has the form $F(k)=\{1,2,\ldots,g(k)\}$ for some $g(k)$.

\Mem{} itself is derived from a closely related game \textsc{Mnem} proposed by Conway~\cite{Conway08}, which is discussed below in Section \ref{section:openquestions}. \Memo{} has also been considered previously: it appears as \#22 in Guy and Nowakowski's 2002 list of unsolved problems~\cite{GN02}, where it is called \textsc{Short Local Nim}. The sequence of frontier values appears in the Online Encyclopedia of Integer Sequences as A131469~\cite{SloaneA131469}.

A substantial amount of work has been done on other games closely related to \Memo. For instance, Chapter 15 in Volume 3 of \emph{Winning Ways}~\cite{BCG03} contains a discussion of the game \textsc{D.U.D.E.N.E.Y.}\footnote{\textsc{D.U.D.E.N.E.Y.} is a rather contrived acronym for \textsc{Deductions Unfalling, Disallowing Echoes, Not Exceeding Y}.} For a fixed value of $Y$, moves in \textsc{D.U.D.E.N.E.Y.} are the same as those of \Memo, except that no more than $Y$ stones may ever be removed on a single turn. The discussion in~\cite{BCG03} refers back to earlier work by Schuh, who discusses the game in~\cite[Chapter XII, \S217--224]{Schuh68} and describes winning strategies when $Y=3,5,7,9$ (the case where $Y$ is even is trivial, since the usual strategy for subtraction games still works).





\section{Nim values of \Memplus}
\label{section:memplus}

The simplest of the three games to understand is \Memplus. See Table~\ref{tab:grundymem+} and Figure~\ref{fig:mem5} for the first few nim values.

\begin{table} \resizebox{.8\textwidth}{!}{\begin{tabular}{c||CCCCC|CCCCC|CCCCC|CCCCC} $n\backslash k$ & 1 & 2 & 3 & 4 & 5 & 6 & 7 & 8 & 9 & 10 & 11 & 12 & 13 & 14 & 15 & 16 & 17 & 18 & 19 & 20 \\ \hline 1 & 0 & 0 & 0 & 0 & 0 & 0 & 0 & 0 & 0 & 0 & 0 & 0 & 0 & 0 & 0 & 0 & 0 & 0 & 0 & 0 \\ 2 & 1 & 0 & 0 & 0 & 0 & 0 & 0 & 0 & 0 & 0 & 0 & 0 & 0 & 0 & 0 & 0 & 0 & 0 & 0 & 0 \\ 3 & 1 & 1 & 0 & 0 & 0 & 0 & 0 & 0 & 0 & 0 & 0 & 0 & 0 & 0 & 0 & 0 & 0 & 0 & 0 & 0 \\ 4 & 1 & 1 & 1 & 0 & 0 & 0 & 0 & 0 & 0 & 0 & 0 & 0 & 0 & 0 & 0 & 0 & 0 & 0 & 0 & 0 \\ 5 & 2 & 1 & 1 & 1 & 0 & 0 & 0 & 0 & 0 & 0 & 0 & 0 & 0 & 0 & 0 & 0 & 0 & 0 & 0 & 0 \\ 6 & 2 & 1 & 1 & 1 & 1 & 0 & 0 & 0 & 0 & 0 & 0 & 0 & 0 & 0 & 0 & 0 & 0 & 0 & 0 & 0 \\ 7 & 2 & 2 & 1 & 1 & 1 & 1 & 0 & 0 & 0 & 0 & 0 & 0 & 0 & 0 & 0 & 0 & 0 & 0 & 0 & 0 \\ 8 & 2 & 2 & 1 & 1 & 1 & 1 & 1 & 0 & 0 & 0 & 0 & 0 & 0 & 0 & 0 & 0 & 0 & 0 & 0 & 0 \\ 9 & 3 & 2 & 2 & 1 & 1 & 1 & 1 & 1 & 0 & 0 & 0 & 0 & 0 & 0 & 0 & 0 & 0 & 0 & 0 & 0 \\ 10 & 3 & 2 & 2 & 1 & 1 & 1 & 1 & 1 & 1 & 0 & 0 & 0 & 0 & 0 & 0 & 0 & 0 & 0 & 0 & 0 \\ 11 & 3 & 2 & 2 & 2 & 1 & 1 & 1 & 1 & 1 & 1 & 0 & 0 & 0 & 0 & 0 & 0 & 0 & 0 & 0 & 0 \\ 12 & 3 & 3 & 2 & 2 & 1 & 1 & 1 & 1 & 1 & 1 & 1 & 0 & 0 & 0 & 0 & 0 & 0 & 0 & 0 & 0 \\ 13 & 3 & 3 & 2 & 2 & 2 & 1 & 1 & 1 & 1 & 1 & 1 & 1 & 0 & 0 & 0 & 0 & 0 & 0 & 0 & 0 \\ 14 & 4 & 3 & 2 & 2 & 2 & 1 & 1 & 1 & 1 & 1 & 1 & 1 & 1 & 0 & 0 & 0 & 0 & 0 & 0 & 0 \\ 15 & 4 & 3 & 3 & 2 & 2 & 2 & 1 & 1 & 1 & 1 & 1 & 1 & 1 & 1 & 0 & 0 & 0 & 0 & 0 & 0 \\ 16 & 4 & 3 & 3 & 2 & 2 & 2 & 1 & 1 & 1 & 1 & 1 & 1 & 1 & 1 & 1 & 0 & 0 & 0 & 0 & 0 \\ 17 & 4 & 3 & 3 & 2 & 2 & 2 & 2 & 1 & 1 & 1 & 1 & 1 & 1 & 1 & 1 & 1 & 0 & 0 & 0 & 0 \\ 18 & 4 & 4 & 3 & 3 & 2 & 2 & 2 & 1 & 1 & 1 & 1 & 1 & 1 & 1 & 1 & 1 & 1 & 0 & 0 & 0 \\ 19 & 4 & 4 & 3 & 3 & 2 & 2 & 2 & 2 & 1 & 1 & 1 & 1 & 1 & 1 & 1 & 1 & 1 & 1 & 0 & 0 \\ 20 & 5 & 4 & 3 & 3 & 2 & 2 & 2 & 2 & 1 & 1 & 1 & 1 & 1 & 1 & 1 & 1 & 1 & 1 & 1 & 0 \end{tabular}}

\vspace{2ex}
\caption{Nim values of \Memplus.} \label{tab:grundymem+} \end{table}

\begin{figure}
  \centering{
  \includegraphics[width=0.4\textwidth]{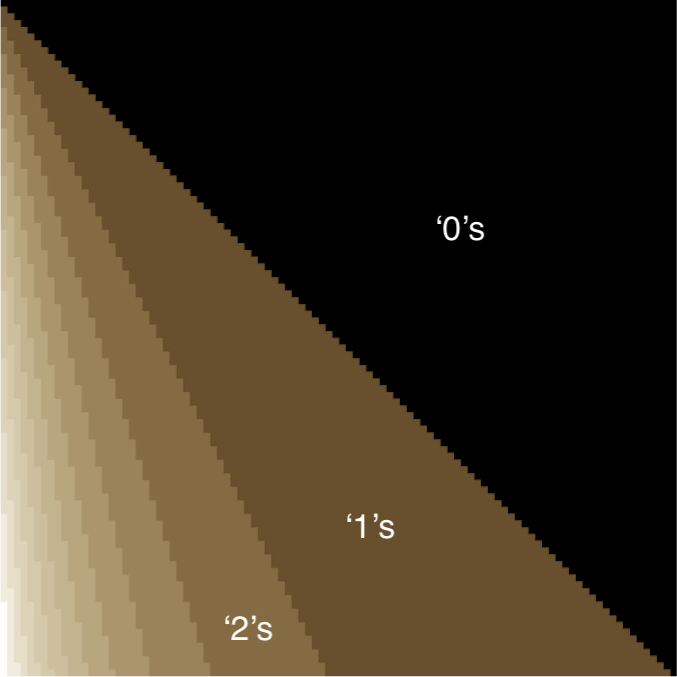}}\caption{The nim values of the game \Memplus ; the columns removal numbers $k$ and the rows heap sizes $n$, with the upper left corner $n_k=1_1$. Note that there is no move from this position, so $\msG(1_1)=0$, expanding into the black region, and the lighter shades symbolize increasing nim values.}\label{fig:mem5}
\end{figure}

\begin{thm} In the game of \Memplus, $\msG(n_k)$ is the largest integer $m$ for which 
\begin{align}\label{eq:Tm}
mk+\frac{m(m+1)}{2}\le n.
\end{align}
\end{thm}


\begin{proof} 
Denote by $T_m$ the $m^{\text{th}}$ triangular number, $T_m = m(m+1)/2$. We define the \emph{$m$-front} to be the set of positions $F_m=\{F_m(k)\}$, where $F_m(k)=(km+T_m)_k$. We define the \emph{$m$-sector} to be the space between the $m$-front (including the $m$-front) and the $(m+1)$-front, i.e.\ $\Delta_m=\bigcup_k\Delta_m(k)$, where \[\Delta_m(k) = \{(km+ T_m)_k,(km+T_m+1)_k,\ldots, (k(m+1)+ T_{m+1}-1)_k\}.\]
For example, the region named `2's in Figure~\ref{fig:mem5} is $\Delta_2$. Observe that $|\Delta_m(k)| = k+m+1$.

We will prove by induction on $n$ that $n_k\in \Delta_j$ if and only if $\msG(n_k)=j$, for all $0\le j<m$; the base case $n=0$ is obvious. 
In order to prove this, it suffices to justify the following claim.

\begin{claim} For $n\in \Delta_m(k)$, then the set $\{\msG((n-k-i)_{k+i})\mid 1\le i\le n-k\}=\{0,\ldots,m-1\}$. \end{claim}

First we prove that, for all $i$, $(n-k-i)_{k+i}\not\in\Delta_m(k+i)$. This follows because \[\min\Delta_m(k+i) = \min\Delta_m(k+1) = \max\Delta_m(k)-k,\] so whenever a player removes more than the column number (here $k+i$) from the $m$-sector, then the resulting position is in an $m'$-sector with $m'<m$. Now we must prove that each such $m'$-sector appears. Firstly, note that $(n-k-1)_{k+1}\in\Delta_{m-1}(k+1)$, whenever $n\in\Delta_m$. Thus, it suffices to show that, for all $0\le j<m-1$, $$\{(n-k-1-i)_{k+1+i}\mid 1\le i <n-k-1\}\cap\Delta_j\ne\varnothing.$$
But this holds, because, for any front position $x_y\in F_{j+1}$, we have $(x-1)_{y+1}\in\Delta_j$, and clearly, for each $j$, exactly one such pair of positions will be obtained as $i$ ranges in the given interval.
\end{proof}

\section{Nim values of \Mem}

The ruleset of \Mem\ is very similar to that of \Memplus. This might lead us to believe that its nim values should be closely related. Indeed, this is true, although there are also some surprises. See Table~\ref{tab:memgrundy} as well as Figure~\ref{fig:memgrundy} for some nim values of \Mem.

\begin{table}
\renewcommand{\oldstylenums}[1]{#1}
\resizebox{.8\textwidth}{!}{\begin{tabular}{c||CCCCC|CCCCC|CCCCC|CCCCC} $n\backslash k$ & 1 & 2 & 3 & 4 & 5 & 6 & 7 & 8 & 9 & 10 & 11 & 12 & 13 & 14 & 15 & 16 & 17 & 18 & 19 & 20 \\ \hline 1 & 1 & 0 & 0 & 0 & 0 & 0 & 0 & 0 & 0 & 0 & 0 & 0 & 0 & 0 & 0 & 0 & 0 & 0 & 0 & 0 \\\cline{2-2} 2 & \multicolumn{1}{c|}{\oldstylenums{2}} & 1 & 0 & 0 & 0 & 0 & 0 & 0 & 0 & 0 & 0 & 0 & 0 & 0 & 0 & 0 & 0 & 0 & 0 & 0 \\ 3 & \multicolumn{1}{c|}{\oldstylenums{1}} & 1 & 1 & 0 & 0 & 0 & 0 & 0 & 0 & 0 & 0 & 0 & 0 & 0 & 0 & 0 & 0 & 0 & 0 & 0 \\ 4 & \multicolumn{1}{c|}{\oldstylenums{2}} & 2 & 1 & 1 & 0 & 0 & 0 & 0 & 0 & 0 & 0 & 0 & 0 & 0 & 0 & 0 & 0 & 0 & 0 & 0 \\ \cline{3-3} 5 & \oldstylenums{3} & \multicolumn{1}{c|}{\oldstylenums{2}} & 1 & 1 & 1 & 0 & 0 & 0 & 0 & 0 & 0 & 0 & 0 & 0 & 0 & 0 & 0 & 0 & 0 & 0 \\  6 & \oldstylenums{4} & \multicolumn{1}{c|}{\oldstylenums{3}} & 2 & 1 & 1 & 1 & 0 & 0 & 0 & 0 & 0 & 0 & 0 & 0 & 0 & 0 & 0 & 0 & 0 & 0 \\  7 & \oldstylenums{3} & \multicolumn{1}{c|}{\oldstylenums{3}} & 2 & 1 & 1 & 1 & 1 & 0 & 0 & 0 & 0 & 0 & 0 & 0 & 0 & 0 & 0 & 0 & 0 & 0 \\  8 & \oldstylenums{2} & \multicolumn{1}{c|}{\oldstylenums{2}} & 2 & 2 & 1 & 1 & 1 & 1 & 0 & 0 & 0 & 0 & 0 & 0 & 0 & 0 & 0 & 0 & 0 & 0 \\  9 & \oldstylenums{4} & \multicolumn{1}{c|}{\oldstylenums{4}} & 3 & 2 & 1 & 1 & 1 & 1 & 1 & 0 & 0 & 0 & 0 & 0 & 0 & 0 & 0 & 0 & 0 & 0 \\ \cline{4-4} 10 & \oldstylenums{3} & \oldstylenums{3} & \multicolumn{1}{c|}{\oldstylenums{3}} & 2 & 2 & 1 & 1 & 1 & 1 & 1 & 0 & 0 & 0 & 0 & 0 & 0 & 0 & 0 & 0 & 0 \\  11 & \oldstylenums{5} & \oldstylenums{3} & \multicolumn{1}{c|}{\oldstylenums{3}} & 2 & 2 & 1 & 1 & 1 & 1 & 1 & 1 & 0 & 0 & 0 & 0 & 0 & 0 & 0 & 0 & 0 \\ 12 & \oldstylenums{4} & \oldstylenums{4} & \multicolumn{1}{c|}{\oldstylenums{4}} & 3 & 2 & 2 & 1 & 1 & 1 & 1 & 1 & 1 & 0 & 0 & 0 & 0 & 0 & 0 & 0 & 0 \\ 13 & \oldstylenums{5} & \oldstylenums{4} & \multicolumn{1}{c|}{\oldstylenums{4}} & 3 & 2 & 2 & 1 & 1 & 1 & 1 & 1 & 1 & 1 & 0 & 0 & 0 & 0 & 0 & 0 & 0 \\ 14 & \oldstylenums{6} & \oldstylenums{5} & \multicolumn{1}{c|}{\oldstylenums{4}} & 3 & 2 & 2 & 2 & 1 & 1 & 1 & 1 & 1 & 1 & 1 & 0 & 0 & 0 & 0 & 0 & 0 \\ 15 & \oldstylenums{3} & \oldstylenums{3} & \multicolumn{1}{c|}{\oldstylenums{3}} & 3 & 3 & 2 & 2 & 1 & 1 & 1 & 1 & 1 & 1 & 1 & 1 & 0 & 0 & 0 & 0 & 0 \\ 16 & \oldstylenums{6} & \oldstylenums{6} & \multicolumn{1}{c|}{\oldstylenums{5}} & 4 & 3 & 2 & 2 & 2 & 1 & 1 & 1 & 1 & 1 & 1 & 1 & 1 & 0 & 0 & 0 & 0 \\ \cline{5-5} 17 & \oldstylenums{5} & \oldstylenums{5} & \oldstylenums{5} & \multicolumn{1}{c|}{\oldstylenums{4}} & 3 & 2 & 2 & 2 & 1 & 1 & 1 & 1 & 1 & 1 & 1 & 1 & 1 & 0 & 0 & 0 \\ 18 & \oldstylenums{4} & \oldstylenums{4} & \oldstylenums{4} & \multicolumn{1}{c|}{\oldstylenums{4}} & 3 & 3 & 2 & 2 & 2 & 1 & 1 & 1 & 1 & 1 & 1 & 1 & 1 & 1 & 0 & 0 \\ 19 & \oldstylenums{6} & \oldstylenums{4} & \oldstylenums{4} & \multicolumn{1}{c|}{\oldstylenums{4}} & 3 & 3 & 2 & 2 & 2 & 1 & 1 & 1 & 1 & 1 & 1 & 1 & 1 & 1 & 1 & 0 \\ 20 & \oldstylenums{7} & \oldstylenums{6} & \oldstylenums{6} & \multicolumn{1}{c|}{\oldstylenums{5}} & 4 & 3 & 2 & 2 & 2 & 2 & 1 & 1 & 1 & 1 & 1 & 1 & 1 & 1 & 1 & 1  \end{tabular}}

\vspace{2ex}
\suppressfloats[t]
\caption{Nim values of \Mem.} \label{tab:memgrundy} \end{table}

%
  
  \begin{figure}[h]
  \centering
  \includegraphics[width=0.3\textwidth]{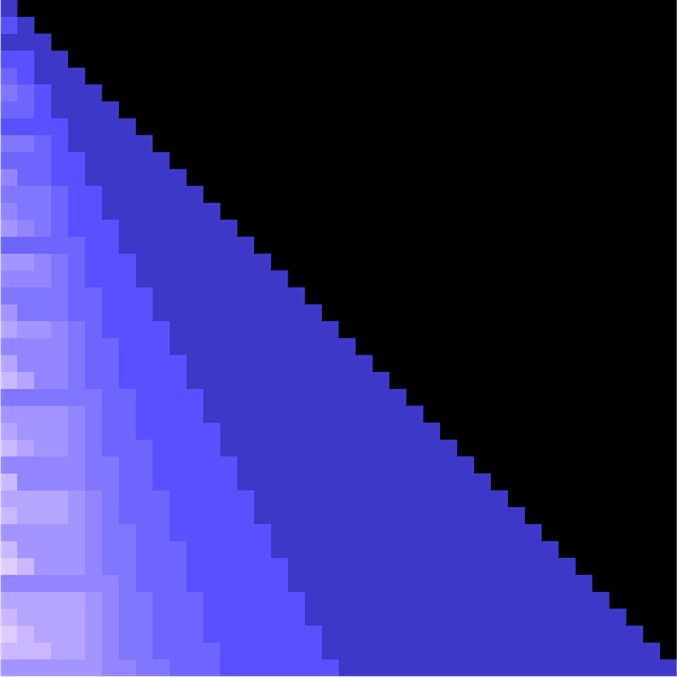}
  \includegraphics[width=0.3\textwidth]{mem3}
  \includegraphics[width=0.3\textwidth]{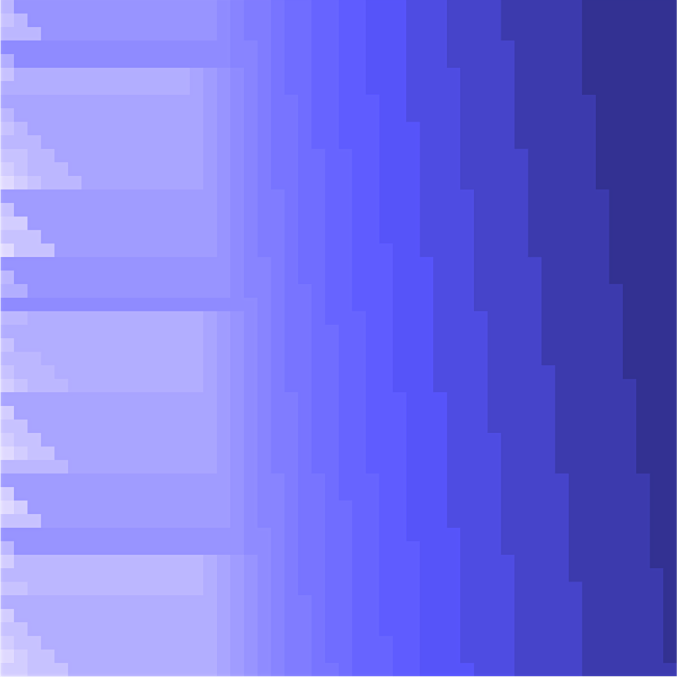}\caption{These pictures show some further nim values of the game \Mem. In the middle picture, on the left, we note the emergence of a parabolic region with high nim values. In the rightmost picture, we have zoomed in on the fractal type behavior inside this region. Each number is a different shade, with lighter shades denoting larger nim values, so black cells are 0, dark blue cells are 1, and so forth.}
  \label{fig:memgrundy}
  \end{figure}

Evidently, there is a lot of structure here. Most of the nim values are indeed very similar to those of \Memplus, but there is a small parabolic region with some more fractal-like behavior. In Table~\ref{tab:memgrundy}, a jagged dividing line on the left hand side of the table delineates the ``regular'' and ``fractal-like'' regions.

\begin{thm}\label{thm:memk2} 
In the game of \Mem, if $k^2\ge n$, then $\msG(n_k)=\lfloor\frac{n}{k}\rfloor$. \end{thm}

\begin{proof} First, note that if $k^2\ge n$, then any move, say to $n'_{k'}$, from $n_k$ we have $k'^2\ge n'$. This is clear, because $k'\ge k$ and $n'<n$, so $k'^2\ge k^2\ge n\ge n'$. Next, we must show, if $k^2\ge n$ and $\lfloor\frac{n}{k}\rfloor=q$, then for any $a$ with $k\le a\le n$, $\lfloor\frac{n-a}{a}\rfloor<q$. This is true because \[\frac{n-a}{a}\le\frac{n-a}{k}\le\frac{n-k}{k}\le\frac{n}{k}-1<\left\lfloor\frac{n}{k}\right\rfloor=q.\] Next, we must show that for every $t$ with $0\le t<q$, there is some integer $a$ with $k\le a\le n$ such that $\left\lfloor\frac{n-a}{a}\right\rfloor=t$. Let us temporarily omit the requirement that $a$ be an integer. The value of $a$ making $\frac{n-a}{a}=t$ is $a=\frac{n}{t+1}$, whereas the value of $a$ making $\frac{n-a}{a}=t+1$ is $a=\frac{n}{t+2}$. It thus suffices to show that there is some \emph{integer} $a$ with \[\frac{n}{t+2} < a \le \frac{n}{t+1}.\] Note that $t<q\le\sqrt{n}$, i.e.\ $t\le\lfloor\sqrt{n}\rfloor-1$. We now consider two cases, which together account for all the possibilities for $t$.
\begin{itemize} 
\item \textbf{Case 1: $t\le \frac{\sqrt{4n+1}-3}{2}$.} We have $(t+1)(t+2)\ge n$ whenever $t\le\frac{\sqrt{4n+1}-3}{2}$. Thus in this case, we have \[\frac{n}{t+1}-\frac{n}{t+2}=\frac{n}{(t+1)(t+2)}\ge 1,\] so there must be an integer $a$ such that \[\frac{n}{t+2} < a \le \frac{n}{t+1}.\] 
\item \textbf{Case 2: $\frac{\sqrt{4n+1}-3}{2}\le t=\lfloor\sqrt{n}\rfloor-1$.} This case occurs when $(t+1)^2\le n< (t+1)(t+2)$. Thus we have \[\frac{n}{t+1}=\frac{n}{\lfloor\sqrt{n}\rfloor}\ge\lfloor\sqrt{n}\rfloor\] and \[\frac{n}{t+2}<t+1=\lfloor\sqrt{n}\rfloor.\] Thus we have \[\frac{n}{t+2}<\lfloor\sqrt{n}\rfloor\le\frac{n}{t+1}.\]
\end{itemize} 
Thus in all cases, there is some integer $n$ such that $\frac{n}{t+2}<a\le\frac{n}{t+1}$. This completes the proof. \end{proof}

\section{$\msP$-positions in \Memo}

The most complex of these three games is \Memo, and it is here that we see the richest structure. See Table~\ref{tab:grundymem0} and Figure~\ref{fig:mem0} for the first few nim values.

\begin{table} \resizebox{.8\textwidth}{!}{\begin{tabular}{c||CCCCC|CCCCC|CCCCC|CCCCC} $n\backslash k$ & 1 & 2 & 3 & 4 & 5 & 6 & 7 & 8 & 9 & 10 & 11 & 12 & 13 & 14 & 15 & 16 & 17 & 18 & 19 & 20 \\ \hline 1 & 0 & 1 & 1 & 1 & 1 & 1 & 1 & 1 & 1 & 1 & 1 & 1 & 1 & 1 & 1 & 1 & 1 & 1 & 1 & 1 \\ 2 & 1 & 1 & 1 & 1 & 1 & 1 & 1 & 1 & 1 & 1 & 1 & 1 & 1 & 1 & 1 & 1 & 1 & 1 & 1 & 1 \\ 3 & 2 & 2 & 0 & 2 & 2 & 2 & 2 & 2 & 2 & 2 & 2 & 2 & 2 & 2 & 2 & 2 & 2 & 2 & 2 & 2 \\ 4 & 2 & 3 & 3 & 0 & 3 & 3 & 3 & 3 & 3 & 3 & 3 & 3 & 3 & 3 & 3 & 3 & 3 & 3 & 3 & 3 \\ 5 & 3 & 3 & 3 & 3 & 0 & 3 & 3 & 3 & 3 & 3 & 3 & 3 & 3 & 3 & 3 & 3 & 3 & 3 & 3 & 3 \\ 6 & 2 & 2 & 2 & 2 & 2 & 2 & 2 & 2 & 2 & 2 & 2 & 2 & 2 & 2 & 2 & 2 & 2 & 2 & 2 & 2 \\ 7 & 4 & 4 & 4 & 4 & 4 & 4 & 0 & 4 & 4 & 4 & 4 & 4 & 4 & 4 & 4 & 4 & 4 & 4 & 4 & 4 \\ 8 & 4 & 5 & 3 & 5 & 5 & 5 & 5 & 5 & 5 & 5 & 5 & 5 & 5 & 5 & 5 & 5 & 5 & 5 & 5 & 5 \\ 9 & 5 & 5 & 5 & 5 & 5 & 5 & 5 & 5 & 0 & 5 & 5 & 5 & 5 & 5 & 5 & 5 & 5 & 5 & 5 & 5 \\ 10 & 6 & 6 & 4 & 6 & 6 & 3 & 6 & 6 & 6 & 6 & 6 & 6 & 6 & 6 & 6 & 6 & 6 & 6 & 6 & 6 \\ 11 & 6 & 5 & 7 & 4 & 7 & 7 & 7 & 7 & 7 & 7 & 0 & 7 & 7 & 7 & 7 & 7 & 7 & 7 & 7 & 7 \\ 12 & 7 & 7 & 7 & 7 & 4 & 7 & 7 & 7 & 7 & 7 & 7 & 0 & 7 & 7 & 7 & 7 & 7 & 7 & 7 & 7 \\ 13 & 6 & 6 & 6 & 6 & 6 & 6 & 6 & 6 & 6 & 6 & 6 & 6 & 0 & 6 & 6 & 6 & 6 & 6 & 6 & 6 \\ 14 & 4 & 4 & 4 & 4 & 4 & 4 & 4 & 4 & 4 & 4 & 4 & 4 & 4 & 4 & 4 & 4 & 4 & 4 & 4 & 4 \\ 15 & 8 & 8 & 7 & 8 & 8 & 8 & 8 & 8 & 8 & 8 & 8 & 8 & 8 & 8 & 0 & 8 & 8 & 8 & 8 & 8 \\ 16 & 8 & 9 & 6 & 9 & 9 & 9 & 9 & 9 & 9 & 9 & 9 & 9 & 9 & 9 & 9 & 0 & 9 & 9 & 9 & 9 \\ 17 & 9 & 9 & 9 & 9 & 9 & 7 & 9 & 9 & 9 & 9 & 9 & 9 & 9 & 9 & 9 & 9 & 0 & 9 & 9 & 9 \\ 18 & 8 & 8 & 8 & 8 & 8 & 8 & 8 & 8 & 8 & 5 & 8 & 8 & 8 & 8 & 8 & 8 & 8 & 8 & 8 & 8 \\ 19 & 10 & 9 & 10 & 10 & 10 & 10 & 10 & 10 & 10 & 10 & 10 & 10 & 10 & 10 & 10 & 10 & 10 & 10 & 0 & 10 \\ 20 & 10 & 11 & 11 & 11 & 11 & 11 & 11 & 11 & 11 & 11 & 11 & 11 & 11 & 11 & 11 & 11 & 11 & 11 & 11 & 0 \end{tabular}}

\vspace{2ex}
\caption{Nim values of \Memo.} \label{tab:grundymem0} \end{table}

  \begin{figure}
  \centering
  \includegraphics[width=0.4\textwidth]{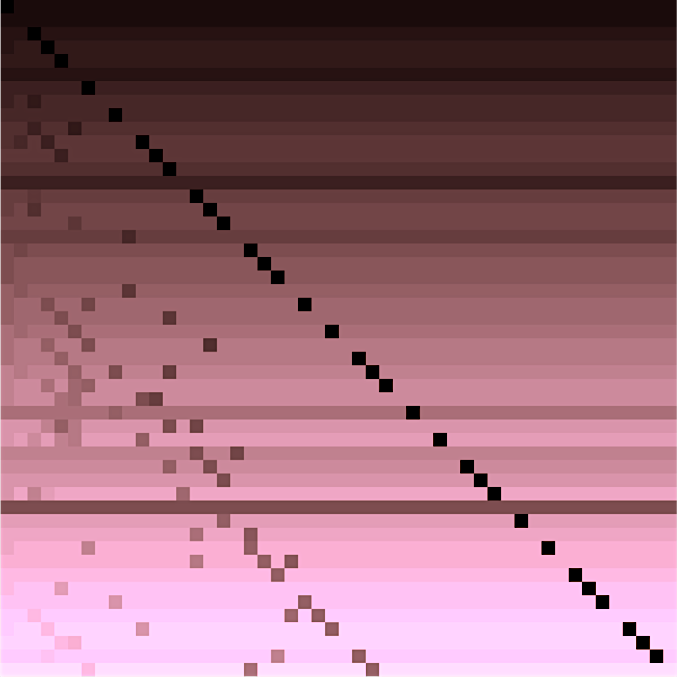}\hspace{2 mm}
  \includegraphics[width=0.4\textwidth]{mem7}\caption{The pictures show the first few nim values of the game \Memo. In the picture to the left, we see in particular the ``0''s on the main diagonal, and in the picture to the right, one can see the emergence of an accompanying ``left-shifted'' diagonal of nim values 12.} \label{fig:mem0}
  \end{figure}

In order to characterize the $\msP$-positions (and higher nim values) of \Memo, we need to introduce the dyadic valuation.

\begin{defn} Let $n$ be a positive integer. We may uniquely write $n=2^em$, where $e$ and $m$ are nonnegative integers and $m$ is odd. We define its \emph{dyadic valuation} to be $v_2(n)=e$. \end{defn}

By convention, we will say that $v_2(0)$ is even, without specifying its value.

\begin{thm} The $\msP$-positions of \Memo\ are of the form $n_n$, where $v_2(n)\equiv 0\pmod{2}$. \end{thm}

\begin{proof}
Consider first a position of the form $n_m$, with $m\ne n$. Then there is a move to $0_0$, so $n_m$ is an $\mcN$ position. Therefore each $\msP$-position must be of the form $n_n$. Note that $v_2(0)\equiv 0\pmod{2}$ by convention and that $0=v_2(1)\equiv 0\pmod{2}$, but $1=v_2(2)\equiv 0\pmod{2}$. Suppose the result holds for all $n'<n$. Now, from $n_n$, with $n>1$ such that $v_2(n)\not\equiv 0\pmod{2}$, clearly $(n/2)_{(n/2)}$ is the desired move option (because $v_2(n/2)\equiv 0\pmod{2}$). On the other hand, if $v_2(n)\equiv 0\pmod{2}$, then either a player must move away from the main diagonal, or leave a position with odd dyadic valuation. \end{proof}

\section{Higher nim values of \Memo}
\label{section:memo_nim_values}

Before computing any higher nim values, we introduce the notion of a \emph{frontier} in \Memo. Note that the positions $n_{n+1},n_{n+2},n_{n+3},\ldots$ all have the same moves available and thus have the same nim values. Thus, we denote each position $n_{n+1},n_{n+2},n_{n+3},\ldots$ by $n_\infty$. We call positions of the form $n_\infty$ \emph{frontier positions}, and we call the nim value $\msG(n_\infty)$ the $n^\text{th}$ \emph{frontier value}.

The first few frontier values are 0,1,1,2,3,3,2,4,5,5,6,7,7,6,4,8,9,9,8,10. We will prove the following results:

\begin{thm} \label{thm:unbounded} The frontier values are unbounded. \end{thm}

\begin{thm} \label{thm:frontier} Every integer appears at least once as a frontier value. \end{thm}

\begin{thm} \label{thm:increasing} If $f(m)$ denotes the least $n$ such that $\msG(n_\infty)=m$, then $f(m)<f(m')$ whenever $m<m'$. \end{thm}

Observe that the positions $n_k$ for $k\le n$ and $n_\infty$ are very similar in structure. They have all the same options, except for one: the position $(n-k)_k$ is an option from $n_\infty$, but not from $n_k$. As a result, we have the following important lemma:

\begin{lemma} \label{lem:twoposs} $\msG(n_k)$ is equal to $\msG(n_\infty)$ or to $\msG((n-k)_k)$. \end{lemma}

\begin{defn} We call $n_k$ an \emph{exceptional position} if $\msG(n_k)=\msG((n-k)_k)$. \end{defn}

As a consequence, we immediately have the following proposition:

\begin{prop} \label{prop:frontier} Suppose $n$ is the smallest integer for which there exists some $k$ with $\msG(n_k)=m$. Then $\msG(n_\infty)=m$. \end{prop}

\begin{proof} By Lemma~\ref{lem:twoposs}, either $\msG(n_k)=\msG(n_\infty)$ or $\msG(n_k)=\msG((n-k)_k)$. By minimality of $n$, we exclude the second possibility. \end{proof}

\begin{thm}[The Final Frontier Theorem] \label{thm:finalfrontier} If $\msG(n_\infty)=m$ and $a>2n$, then $\msG(a_\infty)\neq m$. \end{thm}

\begin{proof} If $a>2n$, then there is a move from $a_\infty$ to $n_{a-n}=n_\infty$. Thus $m=\msG(n_\infty)\neq\msG(a_\infty)$. \end{proof}

Thus $2n$ is the \emph{final} (possible) \emph{frontier} for the nim value $m$.

Theorem~\ref{thm:finalfrontier} allows us to prove Theorem~\ref{thm:unbounded}.

\begin{proof}[Proof of Theorem~\ref{thm:unbounded}] By Theorem~\ref{thm:finalfrontier}, each integer only appears finitely many times on the frontier. Thus, there must be infinitely many (and hence unbounded) numbers on the frontier. \end{proof}

In conjunction with Proposition~\ref{prop:frontier}, we can also establish Theorems \ref{thm:frontier} and~\ref{thm:increasing}:

\begin{proof}[Proof of Theorem~\ref{thm:frontier}] Since the frontier values are unbounded, every integer must appear as some nim value $\msG(n_k)$. By Proposition~\ref{prop:frontier}, the first time $m$ appears as $\msG(n_k)$, it establishes itself on the frontier. Thus every nonnegative integer appears on the frontier. \end{proof}

\begin{proof}[Proof of Theorem~\ref{thm:increasing}] If $m<m'$, then the first instance of $m$ must occur before the first instance of $m'$.  Thus by Proposition~\ref{prop:frontier}, the first frontier value equal to $m$ must be less than the first frontier value equal to $m'$. \end{proof}

What happens to a nim value after the final frontier? We say that a nim value $m$ is \emph{mortal} if $\msG(n_k) = m$ for just finitely many~$n_k$, and otherwise $m$ is \emph {immortal}. It turns out that there is a curious dichotomy here:

\begin{thm}[The Mortality Theorem] \label{thm:mortality} Suppose that $m$ appears at least twice on the frontier, say as $\msG(n_\infty)=\msG(n'_\infty)=m$ with $n<n'$. Then if $a>2n'$, we have $\msG(a_k)\neq m$ for all $k$. Thus the value $m$ dies out after row $2n'$. \end{thm}

\begin{proof} If $a>2n'$, then from $a_\infty$, there are moves to both $n_{a-n}=n_\infty$ and to $n'_{a-n'}=n'_\infty$. From $a_k$, at least one of these is a legal move: the legal moves are to $(a-i)_i$ for $i\neq k$, and $k$ cannot be equal to both $n$ and $n'$ simultaneously. Thus $m$ is an excludant for $a_k$, so $\msG(a_k)\neq m$. \end{proof}

\begin{ex} Let $m=11$. The first frontier value for $m$ is $n=20$, so that $\msG(20_\infty)=11$. There is also a second frontier value of 11, namely $\msG(21_\infty)=11$. Thus Theorem~\ref{thm:mortality} implies that 11 never appears as a nim value of $a_k$ for $a>42$. It turns out that there are several additional nim values equal to 11 with $21 < a \leq 42$, namely $\msG(22_2)=\msG(40_{19})=\msG(42_{22})=11$. \end{ex}

The Mortality Theorem allows for the possibility that a number can appear exactly once along the frontier. Indeed, this happens with $m=0$: we have $\msG(0_\infty)=0$, but 0 does not occur again along the frontier. When a number occurs only once on a frontier, then it does \emph{not} die out at any point. Indeed, there are arbitrarily large values of $n$ for which $\msG(n_n)=0$.

It turns out that whenever a nim value $m$ occurs exactly once along the frontier, it exhibits a very similar pattern to that of~$0$: for sufficiently large~$n$, we have $\msG(n_k) = m$ only on the diagonal $n = t + k$, where $t$ is the unique integer satisfying $\msG(t_\infty) = m$. Moreover, the values of $k$ for which $\msG((t+k)_k) = m$ can by characterized in terms of the dyadic valuation of~$k$. The following Theorem makes this observation precise. (The next occurrence of this phenomenon after $m = 0$ is $m = 12$, discussed in more detail below.)

\begin{thm}
\label{thm:diagonal}
Suppose that nim value $m$ occurs exactly once on the frontier, say ${\msG(t_\infty) = m}$. Then for all $n > 2t$ with $\msG(n_k) = m$, we necessarily have $n = t + k$. Moreover, for all $k > 3t$,
\[
\msG((t+k)_k) = m \quad\text{if and only if}\quad
\begin{cases}
\text{$k$ is odd; or} \\
\text{$k$ is even and } \msG((t+k/2)_{k/2}) \neq m.
\end{cases}
\]
\end{thm}

\begin{proof}
First suppose $n > 2t$. Then for all $k \neq n-t$, there is a move from $n_k$ to $t_{n-t} = t_\infty$. Therefore $m = \msG(t_\infty)$ is an excludant~of $n_k$, so necessarily $\msG(n_k) \neq m$. This proves the first assertion.

For the second part of the Theorem, suppose $k > 3t$ and consider the options of $(t+k)_k$. Certainly $(t+k)_k$ has moves to $0_\infty$, $1_\infty$, $\ldots$, $(t-1)_\infty$, and therefore, by Theorem~\ref{thm:increasing}, each of $0$, $1$, $\ldots$, $m-1$ is an excludant. So $\msG((t+k)_k) = m$ if and only if $m$ is not {\it also} an excludant of $(t+k)_k$.

Now if $k$ is even, then $(t+k)_k$ has exactly one option that remains on the $n = t+k$ diagonal: $(t+k/2)_{k/2}$. If $k$ is odd, then every option falls outside the diagonal. So consider any option $a_{t+k-a}$ that is not on the diagonal. If $a < t+k-a$, then $a_{t+k-a} = a_\infty$, and hence $\msG(a_{t+k-a}) \neq m$. Otherwise, $2a \geq t + k > 4t$, so that $a > 2t$. By the first part of the Theorem, we also have $\msG(a_{t+k-a}) \neq m$. This shows that there are no options of value $m$ outside the diagonal, which in turn proves the Theorem.
\end{proof}

Now suppose $k > 3t$ in Theorem~\ref{thm:diagonal}, and write $k = b \cdot 2^e$, with $b$ odd. If $b > 3t$, then Theorem~\ref{thm:diagonal} shows that
\begin{equation}
\msG((t+k)_k) = m \quad\text{if and only if}\quad v_2(k) \text{ is even}.
\label{eqn:dyadicvaluationrelation}
\end{equation}
So the relationship between the values $\msG((t+k)_k)$ and the valuations $v_2(k)$ is the same as for the $m = 0$ case, except along finitely many values of~$b$.

Furthermore, if $b \leq 3t$, then there is a unique integer $e$ for which $3t/2 < b \cdot 2^e \leq 3t$. Then if the relationship in equation (\ref{eqn:dyadicvaluationrelation}) holds for $k = b \cdot 2^e$, it also holds for $b \cdot 2^{e'}$, for all $e' \geq e$. If (\ref{eqn:dyadicvaluationrelation}) fails for~$k$, then in fact the converse relationship holds for all $e' \geq e$:
\[
\msG((t+k)_k) = m \quad\text{if and only if}\quad v_2(k) \text{ is odd}.
\]

\begin{ex} 
The first nonzero immortal value is $m = 12$, which (by rote computation) occurs on the frontier for the first time at $22_\infty$, but not at any $n_\infty$ for $22 < n \leq 44$. By Theorem~\ref{thm:diagonal}, we know that for $n > 44$, nim value $12$ occurs only on the diagonal $n = 22 + k$.

In fact, the only occurrences of $m = 12$ outside this diagonal are at $24_1$, $32_5$, and $22_k$ for $k \neq 2,10$. Moreover, by direct computation of the values $(22+k)_k$ for $k \leq 66$, together with Theorem~\ref{thm:diagonal}, we can give a complete characterization of the diagonal. $\msG((22+k)_k) = 12$ if and only if $v_2(k)\equiv 0\pmod{2}$, with the following exceptions: \begin{itemize} \item $k=2^e$, $e\ge 4$, \item $k=3\times 2^e$, $e\ge 0$, \item $k=15\times 2^e$, $e\ge 0$. \end{itemize} For these exceptional cases, $\msG((k+22)_k)=12$ iff $v_2(k)\equiv 1\pmod{2}$.

\end{ex}

A calculation of the values $\msG(n_k)$ for $n \leq 500000$ reveals two additional immortal values: $m = 1270$ and $m = 105161$. As with $m = 0$ and $m = 12$, their diagonals can be characterized by a finite ``signature'': a finite list of exceptional values of~$b$, for which the asymptotic behavior is inverted; and a finite list of exceptional ``early'' values of~$k$ that violate the asymptotic rule.

For $m = 1270$, the relevant diagonal is $(2782+k)_k$, the asymptotic exceptions occur at $b \in X$, where $X = \{3, 19, 27, 45, 143, 477, 2067, 2091\}$, and the only early exceptions are $k = 143$, $286$, $572$, and~$1144$. Therefore, $\msG((2782+k)_k) = 1270$ if and only if, writing $k = b \cdot 2^e$, either:
\begin{itemize}
    \item $b \not\in X$ and $v_2(k) \equiv 0\pmod{2}$; or
    \item $b \in X$, $k \neq 143,286,572,1144$, and $v_2(k) \equiv 1\pmod{2}$; or
    \item $k = 143,286,572,1144$ and $v_2(k) \equiv 0\pmod{2}$.
\end{itemize}

For $m = 105161$, there are $106$ asymptotic exceptions and $143$ early exceptions; we omit the full presentation. See Appendix~\ref{appendix:algorithms} for a description of the algorithm used to calculate these values.

The sequence of immortal values $m = 0,12,1270,105161,\ldots$ appears to be a new integer sequence. It has been entered into the Online Encyclopedia of Integer Sequences as sequence A351630~\cite{SloaneA351630}.

\suppressfloats[t]
\section{Questions}
\label{section:openquestions}
Memgames are mysterious, and we have many more questions than answers. For starters, there is the ``fractal-like'' region below the parabolic envelope in \Mem. Figure~\ref{fig:mem_closeup} shows a small region of the heat map of \Mem, with $55 \leq n \leq 85$ and $1 \leq k \leq 20$.

Tantalizing patterns can be ascertained: particular nim values tend to establish themselves as ``dominant'' within various triangular subregions, carving out overlapping upper triangles of constant value. Each nim value $m$ appears to make its last ``fractal-like'' (below the parabola) appearance at row $n = m(m+2)$, where it is especially dominant: $\msG(m(m+2)_k) = m$ for all $k \leq m + 2$, occupying the entirety of the fractal-like region on row $m(m+2)$. (The rows $n = 7 \cdot 9$ and $n = 8 \cdot 10$ are visible in Figure~\ref{fig:mem_closeup}, along with other ``dominated'' rows that do not follow the same pattern.) We do not yet have proofs, but the curious reader will be led to discover these patterns, and many more, buried within the fine structure of \Mem.

\begin{figure}[h]
    \centering
    \includegraphics{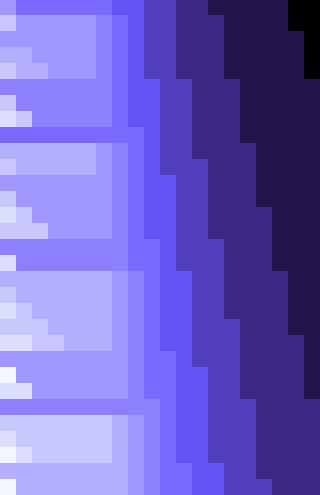}
    \caption{Closeup view of the nim values of \Mem, in the region $55 \leq n \leq 85$, $1 \leq k \leq 20$.}
    \label{fig:mem_closeup}
\end{figure}

Conway's original game \textsc{Mnem}, which motivated the study of \Mem, is still more mysterious. From a \textsc{Mnem} heap~$n_k$, it is legal to remove $j$ tokens for $j \geq k$, just as in \Mem; but it is also legal to {\it add} $j$ tokens for any $j < k$, so that
\[
n_k = \{(n-j)_j : k \leq j \leq n\} \cup \{(n+j)_j : 1 \leq j < k\},
\]
and only the position $0_1$ is terminal.
Therefore \textsc{Mnem} is ostensibly loopy, and there are sequences that cause the heaps to grow unboundedly:
\[
5_1 \to 1_4 \to 4_3 \to 6_2 \to 7_1 \to 1_6 \to 6_5 \to 10_4 \to 13_3 \to 15_2 \to 16_1 \to 1_{15} \to 15_{14} \to \cdots
\]
However, it is conjectured that all nim values are finite, and this conjecture has been verified for $n \leq 1000$. Indeed, the structure of the nim values of \textsc{Mnem} appears similar to \Mem, with the same parabolic envelope and a similar (but curiously, not identical) fractal-like interior. It appears that it is {\it sometimes} necessary to grow a heap in order to win; but with best play no heap will grow unboundedly: either $\mathscr{P}$ or $\mathscr{N}$ can force a win in finite time. Yet even this has not been proved.

Finally, we can ask a host of questions about \Memo.
\begin{itemize}
\item \Memo\ has simple, parameterless rules, yet they give rise to the unusual and mysterious sequence of immortal values $m = 0,12,1270,105161,\ldots.$ What is special about these numbers, or are they merely a combinatorial accident?
\item Are there infinitely many immortal nim values?
\item How many times can a nim value $m$ appear on the frontier? We have found values that appear four times on the frontier, the smallest of which is $m = 871$. We conjecture that there exist nim values that appear arbitrarily many times on the frontier.
\item Are there generalizations of \Memo{} for which we can prove a general theory about frontiers and diagonals, but that exhibit other behaviors, such as diagonals other than $m = 0,12,1270,105161,\ldots$? (For example, one might consider ``perturbed'' variants of \Memo{} in which additional moves are permitted on small heap sizes, but the limiting behavior is the same.)
\item Are there other memgames whose nim values have interesting structure? Specifically, in \cite{LRS09}, the memory is extended to include the $k$ previous moves by the other player, where $k$ is a ruleset parameter, and it is surprisingly demonstrated that the games have the same $\mathscr{P}$/$\mathscr{N}$ structure as games with a certain ``$k$-blocking maneuver.'' In our setting, how do the nim values change if we extend the definition of \Memo\ to allow up to $k-1$ consecutive mimics of the other player's move, but not the $k^{\rm th}$ one? For yet another variation, one may want to study the game where the $k$ previous move sizes (by either player) are not allowed.
\item In \cite{LR18} the memory function of {\sc Fibonacci Nim} is extended to range over all heaps, to a global parameter: Study {\sc Global Mem}. 
\end{itemize}

There is undoubtedly much more to discover in this strange and fascinating landscape.
 
\appendix
\section{An Algorithm for Calculating Values of \Memo \label{appendix:algorithms}}

Evaluating $\msG(n_k)$ in \Memo{} is ostensibly an $O(n^3)$ calculation in $O(n^2)$ space, but with some simple optimizations the running time can be reduced substantially, to $O(n^2\cdot e)$ in $O(n \cdot e)$ space, where $e$ is the typical number of exceptions per row. (An \emph{exception} is a value of $k$ for which $\msG(n_k) \neq \msG(n_\infty)$.)

The algorithm used for the calculations in this paper is given as follows. For each row~$n$, we store the frontier value~$n_\infty$, together with the finite list of exceptions $(k,\ \msG(n_k))$. Then assuming rows $0,\ldots,n-1$ have been computed and stored, we compute row $n$ as follows.
\begin{itemize}
    \item First, iterate over $\msG((n-k)_k)$ for $1 \leq k < n$ and, for each possible nim value $m < n$, tabulate the number of observations of~$m$, say~$\chi(m)$. The smallest $m$ for which $\chi(m) = 0$ is the frontier value $\msG(n_\infty)$. Retain the table $m \mapsto \chi(m)$ for the next step.
    \item Next, iterate once again over $k$ for $1 \leq k < n$. For each~$k$, let $m = \msG((n-k)_k)$. If $m > \msG(n_\infty)$ \emph{or} if $\chi(m) \geq 2$, then necessarily $\msG(n_k) = \msG(n_\infty)$, so no action is necessary. If $m < \msG(n_\infty)$ \emph{and} $\chi(m) = 1$, then necessarily $\msG(n_k) = m$, so add the pair $(k,m)$ to the list of exceptions for row~$n$.
\end{itemize}
By computing the table $m \mapsto \chi(m)$ and retaining it throughout the calculation of row~$n$, we ensure that only a single mex operation is needed per row.
 
\section*{Acknowledgements}

This work was started at Games at Dal, at Dalhousie University, in 2015.

\bibliography{memgames}
\bibliographystyle{halpha}
\end{document}